\documentclass[preprint]{elsarticle}

\usepackage{lineno,hyperref}
\modulolinenumbers[1]

\usepackage{amsmath,amsthm,amssymb}
\usepackage{kbordermatrix}
\usepackage{enumitem}
\usepackage{physics}
\usepackage{mathtools}
\usepackage{multirow}                   
\usepackage{mathrsfs}                   

\allowdisplaybreaks

\newtheorem{thm}{Theorem}[section]
\newtheorem{lem}[thm]{Lemma}
\newtheorem{cor}[thm]{Corollary}

\theoremstyle{definition}			                						
\newtheorem{mydef}[thm]{Definition}
\newtheorem{rem}[thm]{Remark}
\newtheorem{ex}[thm]{Example}

\DeclareMathOperator{\diag}{diag}
\DeclareMathOperator{\circulant}{circ}

\DeclareMathOperator{\nullspace}{Null}

\newcommand{\ii}{\textup{i}}

\journal{Linear Algebra and its Applications}

\begin{document}

\begin{frontmatter}

\title{Jordan chains of $h$-cyclic matrices, II}


\author{Andrew L.~Nickerson}
\ead{andrewlewisnickerson@gmail.com}

\author[uwbaddress]{Pietro Paparella\corref{ca}}
\cortext[ca]{Corresponding author}
\ead{pietrop@uw.edu}
\ead[url]{http://faculty.washington.edu/pietrop/}
\address[uwbaddress]{Division of Engineering and Mathematics, University of Washington Bothell, Bothell, WA 98011-8246, USA}

\begin{abstract}
McDonald and Paparella [\emph{Linear Algebra Appl}.~498 (2016), 145--159] gave a necessary condition on the structure of Jordan chains of $h$-cyclic matrices. In this work, that necessary condition is shown to be sufficient. As a consequence, we provide a spectral characterization of nonsingular, $h$-cyclic matrices. In addition, we provide results for the Jordan chains corresponding to the eigenvalue zero of singular matrices. Along the way, a new characterization of \emph{circulant matrices} is given.
\end{abstract}


\end{frontmatter}


\section{Introduction}

A celebrated result in spectral graph theory states that a graph is bipartite if and only if the spectrum of its adjacency matrix is \emph{symmetric} (i.e., $-\lambda$ is an eigenvalue whenever $\lambda$ is). However, as noted by Nikiforov \cite[p.~3]{n2017}, a bipartite \emph{digraph} can not, in general, be characterized by its spectrum alone and restrictions on the Jordan structure of the matrix must be considered. 

Recall that a \emph{directed graph} (or \emph{digraph}, for brevity) $\Gamma = (V,E)$ consists of a finite, nonempty set $V$ of \emph{vertices}, together with a set $E \subseteq V \times V$ of \emph{arcs}. If $A$ is an $n$-by-$n$ matrix with entries over a field $\mathbb{F}$, then the \emph{digraph of $A$}, denoted by $\Gamma = \Gamma_A$, has vertex set $V = \{ 1, \dots, n \}$ and arc set $E = \{ (i, j) \in V \times V \mid a_{ij} \neq 0\}$. 

A digraph $\Gamma = ( V, E )$ is called \emph{$h$-partite} if there is a partition \(P = \{ V_1,\ldots, V_h \}\) of $V$ such that, for each arc $(i, j) \in E$, there are positive integers $k,\ell \in \{ 1, \dots, h \}$ such that $i \in V_\ell$ and $j \in V_{k}$. A digraph $\Gamma = ( V, E )$ is called \emph{cyclically $h$-partite} if there is a partition $P = \{ V_1,\dots, V_h \}$ of $V$ such that, for each arc $(i, j) \in E$, there is a positive integer $\ell \in \{ 1, \dots, h \}$ such that $i \in V_\ell$ and $j \in V_{\ell+1}$ (where, for convenience, $V_{h + 1} := V_1$). Notice that there is no distinction between a cyclically bipartite digraph and a bipartite digraph, but there is if $h > 2$. Furthermore, notwithstanding its more restrictive nature, characterizing the spectral properties of cyclically $h$-partite digraphs subsumes the case of bipartite digraphs.

A matrix $A$ is called \emph{$h$-cyclic} or \emph{cyclically $h$-partite} if $\Gamma_A$ is cyclically $h$-partite. McDonald and Paparella \cite[Lemma 4.1]{mp2016} showed that if $A$ is a nonsingular, $h$-cyclic matrix with complex entries, then a given Jordan chain corresponding to an eigenvalue $\lambda$ of $A$ determines a Jordan chain corresponding to $\lambda\omega^k$, $1 \le k \le h-1$, where $\omega:=\exp(2\pi \ii/h)$ (for full details, see Lemma \ref{lem:jchainlemma} below). As an immediate consequence, if the nonsingular Jordan block $J_p(\lambda)$ appears in any Jordan canonical form of $A$, then the nonsingular Jordan block $J_p(\lambda \omega^k)$ also appears in any Jordan canonical form of $A$, $\forall k \in \{1,\dots,h-1\}$ \cite[Corollary 4.2]{mp2016}. 

The central aim of this work is to establish a converse of this result---i.e., if a given Jordan chain corresponding to an eigenvalue $\lambda$ of a matrix $A\in \mathsf{M}_n(\mathbb{C})$ determines a Jordan chain chain for $\lambda\omega^k$, for every eigenvalue $\lambda$ of $A$, then $A$ is $h$-cyclic (see Theorem \ref{thm:jchain} below). As a consequence, we provide a spectral characterization of nonsingular, $h$-cyclic matrices. In addition, we provide results for Jordan chains corresponding to the eigenvalue zero of singular matrices and implicitly provide an algorithm on computing these chains. Along the way, we provide a new characterization of \emph{circulant matrices} and use this characterization to provide a more rigorous proof of a result by McDonald and Paparella \cite[Lemma 4.3]{mp2016}. 

\section{Notation and Background}

For $n \in \mathbb{N}$, denote by: $\langle n \rangle$ the set $\{1,\ldots,n\}$; $R(n)$ the set $\{0,1,\ldots,n-1\}$; and $\omega = \omega_n$ the complex number $\exp(2\pi\ii/n)$.

The algebra of $m$-by-$n$ matrices over a field $\mathbb{F}$ is denoted by $\mathsf{M}_{m \times n}(\mathbb{F})$; when $m = n$, the set $\mathsf{M}_{m \times n}(\mathbb{F})$ is abbreviated to $\mathsf{M}_{n}(\mathbb{F})$. If $A \in \mathsf{M}_{m \times n}(\mathbb{F})$, then the $(i,j)$-entry of $A$ is denoted by $[A]_{ij}$, $a_{ij}$, or $a_{i,j}$. 

If $A \in \mathsf{M}_n(\mathbb{C})$ and $R$, $C \subseteq \{1,\dots, n\}$, then $A(R,C)$ denotes the submatrix of $A$ whose rows and columns are indexed by $R$ and $C$, respectively. 

If $A, B \in \mathsf{M}_{m \times n}(\mathbb{F})$, then the \emph{Hadamard product of $A$ and $B$}, denoted by $A \circ B$, is the $m \times n$ matrix such that $[A \circ B]_{ij} = a_{ij} b_{ij}$.

If $\lambda \in \mathbb{F}$, then the \emph{$n$-by-$n$ Jordan block with eigenvalue $\lambda$}, denoted by $J_n(\lambda)$, is defined by 
\[ J_n(\lambda)  
=\lambda \sum_{i=1}^n E_{ii} + \sum_{i=1}^{n-1} E_{i,i+1} 
= \begin{bmatrix}
\lambda & 1         &                                   \\
        & \lambda   & 1                                 \\
        &           & \ddots    & \ddots                \\
        &           &           & \lambda   & 1         \\
        &           &           &           & \lambda 
\end{bmatrix} \in \mathsf{M}_n(\mathbb{F}), \]
where $E_{ij}$ is the $n$-by-$n$ matrix whose $(i,j)$-entry equals one and zero otherwise.

\subsection{Cyclically $h$-partite matrices}

If $\Gamma_A$ is cyclically $h$-partite with partition $P$, then $A$ is said to be \emph{h-cyclic with partition} $P$ or that $P$ \emph{describes the $h$-cyclic structure of A}. The partition $P = \{V_1,\dots, V_h\}$ is \emph{consecutive} if $V_1 = \{1,\dots, i_1\}$, $V_2 = \{i_1 + 1,\dots, i_2\}, \ldots, V_h = \{ i_{h - 1} + 1, \dots, n \}$. If $A$ is $h$-cyclic with consecutive partition $P$, then $A$ is of the form
\begin{align}
    \label{cyclic_form}
    \kbordermatrix{
            & 1         & 2         & \cdots    & h         \\
    1       &           & A_{12}    &           &           \\
    \vdots  &           &           & \ddots    &           \\
    h-1     &           &           &           & A_{h-1,h} \\
    h       & A_{h1}    &           &           &  
    },       
\end{align}
where $A_{i,i+1} = A(V_i,V_{i+1}) \in \mathsf{M}_{\vert V_i \vert\times\vert V_{i+1}\vert}(\mathbb{C})$, $\forall i \in \langle h \rangle$ \cite[p.~71]{br1991}. If $P$ is not consecutive, then there is a permutation matrix $Q$ such that $Q^\top A Q$ is $h$-cyclic with consecutive partition \cite[p.~71]{br1991}. Given that the eigenvalues of a matrix do not change with respect to permuatition similarity, hereinafter it is assumed, without loss generality, that every $h$-cyclic matrix is of the form \eqref{cyclic_form}. 

If $x \in \mathbb{C}^n$ and 
\[
x = 
\begin{bmatrix}
x_1     \\
\vdots  \\
x_h
\end{bmatrix},
\]
where $x_i \in \mathbb{C}^{\vert V_i \vert}$, then $x$ is said to be \emph{conformably partitioned with $A$ (or $P$)}.

Suppose that $x_1,\dots,x_p \in \mathbb{C}^n$. Recall that if $x_1$ is an eigenvector corresponding to $\lambda \in \mathbb{C}$ and \(Ax_i = \lambda x_i + x_{i-1}\), $1 < i \le p$, then $\{ x_1,\dots,x_p \}$ is called a \emph{Jordan chain of $A$ corresponding to $\lambda$}. Furthermore, it can be shown that if $\{ x_1, \ldots, x_p \}$ is a Jordan chain, then $\{ x_1, \ldots, x_p \}$ is linearly independent and $x_k = (A - \lambda I)^{p - k} x_p$, $1 \le k \le p$. 

The following partial products of submatrices of an $h$-cyclic matrix will be of use in the sequel. 

\begin{mydef}
\label{consecutive-product}
Let $A\in \mathsf{M}_n(\mathbb{C})$ and suppose that $A$ is of the form \eqref{cyclic_form}. For $i \in \langle h\rangle$ and $p \in \mathbb{N}$, let 
\[B_{ip} := \prod_{j=h+1-p}^{h} A_{\alpha^{j-1}(i),\alpha^{j}(i)}, \]
where $\alpha\in S_h$ is the $h$-cycle of order $h$ defined by \( \alpha(i) = i\bmod{h}+ 1\). For ease of notation, the matrix $B_{ih}$ is abbreviated to $B_i$. Notice that $B_i$ is a square matrix of order $\vert V_i \vert$.
\end{mydef}

\begin{rem}
Notice that 
\begin{align*}
    \alpha^{-1} (i) =
    \begin{cases}
    i-1,    & 1 < i \le h \\
    h,      & i = 1.
    \end{cases}
\end{align*}
and $\alpha^j$ is defined if $j < 0$ since $\alpha$ is an invertible map.
\end{rem}

If $A$ is of the form \eqref{cyclic_form}, then 
\begin{align}
A^h = 
\bigoplus_{j=1}^h B_j =
\begin{bmatrix}
B_1                         &           & \multirow{2}{*}{\LARGE 0} \\
\multirow{2}{*}{\LARGE 0}   & \ddots    &                           \\
                            &           & B_h
\end{bmatrix} 
\label{A-to-the-h}
\end{align}
(see, e.g., Brualdi and Ryser \cite[p.~73]{br1991}).

We note the following useful theorem due to Mirsky.

\begin{thm}
[Mirsky {\cite[Theorem 1]{m1966}}]
\label{thm:Minc-result}
Let $A\in \mathsf{M}_n(\mathbb{C})$ and suppose that $A$ is of the form \eqref{cyclic_form}. If $\lambda_1,\ldots,\lambda_m$ are the nonzero eigenvalues of $B_1$, then the spectrum of $A$ consists of $n-hm$ zeros and the $hm$ $h$-th roots of $\lambda_1,\ldots,\lambda_m$.
\end{thm}

\subsection{Circulant Matrices}

We digress to briefly discuss circulant matrices (for a general reference, see Davis \cite{d1979}).

\begin{mydef}
[{\cite[p.~66]{d1979}}]
    \label{circulant-def-1}
If $r  := \begin{pmatrix} r_1,  \ldots, r_n \end{pmatrix}$, where $r_i \in \mathbb{C}$, $i \in \langle n \rangle$, and $C\in \mathsf{M}_n(\mathbb{C})$, then $C$ is called a \emph{circulant} or a \emph{circulant matrix} with \emph{reference vector} $r$, denoted $\circulant(r)$, if
    \[ c_{ij} = r_{((j-i)\bmod{n}) + 1} \]
for every $(i,j) \in \langle n \rangle^2$.
\end{mydef}

\begin{mydef}
Denote by $e_i$ the $i$\textsuperscript{th} canonical basis vector of $\mathbb{C}^n$. For $n \in \mathbb{N}$, $n \ge 2$, the \emph{basic circulant of order $n$}, denoted by $K_n$, is the circulant matrix defined by $K_n = \circulant(e_2)$.

For example, 
\[ 
K_2 =
\begin{bmatrix}
0 & 1 \\
1 & 0 
\end{bmatrix},
\]
\[ 
K_3 =
\begin{bmatrix}
0 & 1 & 0 \\
0 & 0 & 1 \\
1 & 0 & 0 
\end{bmatrix}, 
\]
and, in general,  
\[ K_n =
\begin{bmatrix}
    & I_{n-1} \\
1   & 
\end{bmatrix}.
\]
\end{mydef}

The following result, although unsurprising, provides a characterization of circulant matrices that, to the best of our knowledge, has not appeared previously in the literature.

\begin{thm} 
    \label{circ-char}
If $C \in \mathsf{M}_n(\mathbb{C})$, then $C$ is a circulant matrix if and only if $c_{i_1,j_1} = c_{i_2,j_2}$ for all $(i_1,j_1) \sim (i_2,j_2)$, where $\sim$ is the equivalence relation defined by 
    \begin{align*}
        (i_1,j_1) \sim (i_2,j_2) \iff j_1-i_1 \equiv (j_2-i_2)\bmod{n}. 
    \end{align*}
\end{thm}

\begin{proof}
Suppose that $C = \circulant(r)$. If $(i_1,j_1)\sim(i_2,j_2)$, then  
    \[ c_{i_1,j_1} = r_{((j_1-i_1)\bmod{n}) + 1} = r_{((j_2-i_2)\bmod{n}) + 1} = c_{i_2,j_2}, \]
as desired.

Conversely, suppose that $c_{i_1,j_1} = c_{i_2,j_2}$ for all $(i_1,j_1) \sim (i_2,j_2)$ and let 
\[ r_k := c_{1k},~\forall k \in \langle n \rangle. \] 
If $(i,j) \in \langle n \rangle^2$ and $m := (j-i)\bmod{n}$, then 
    \[ (j-i)\bmod{n} = m\bmod{n} = (m+1-1)\bmod{n}, \]
i.e., $(i,j)\sim(1,m+1)$. Thus,
    \[ c_{ij} = c_{1,m+1} = r_{m+1} = r_{((j-i)\bmod{n})+1} \]
as desired.
\end{proof}

\section{Nonsingular $h$-cyclic Matrices}

We briefly digress to address some properties of nonsingular, $h$-cyclic matrices.

\begin{thm}
\label{equal-size-partition-sets}
Let $A\in \mathsf{M}_n(\mathbb{C})$ be $h$-cyclic and nonsingular. If $P = \{ V_1,\dots, V_h \}$ describes the $h$-cyclic structure of $A$, then $\vert V_i \vert = \vert V_j \vert$, $\forall i,j \in \langle h\rangle$. 
\end{thm}

\begin{proof} 
For contradiction, suppose that $\vert V_i \vert \ne \vert V_j \vert$. Without losing generality, it may be assumed that $\vert V_j \vert > \vert V_i \vert$. Notice that $B_j \in \mathsf{M}_{\vert V_j \vert}(\mathbb{C})$ and 
\begin{align*}
    \rank (B_j) 
    &\leq \rank \left( \prod_{k=1}^{h} A_{\alpha^{k-1}(j),\alpha^{k}(j)} \right)                \\
    &\leq \min \left\{ \rank{(A_{12})}, \ldots, \rank{(A_{h-1,h})},\rank{(A_{h,1})} \right\}    \\
    & \leq \vert V_{i}\vert                                                                     \\
    &< \vert V_j \vert,         
\end{align*} 
i.e., $B_j$ is rank-deficient. Thus, $0\in\sigma(B_j)$, but \[0 \in \sigma(B_j) \implies 0 \in \sigma(A^h) \iff 0 \in \sigma(A)\] which contradicts the assumption that $A$ is non-singular.
\end{proof}

\begin{cor}
If $A\in \mathsf{M}_n(\mathbb{C})$ is $h$-cyclic and nonsingular, then $h$ divides $n$.
\end{cor}

\begin{rem}
The converse of Theorem \ref{equal-size-partition-sets} does not hold; indeed, consider the singular bipartite matrix 
\[
\begin{bmatrix} 
0 & 0 & 1 & 1 \\
0 & 0 & 1 & 1 \\
1 & 1 & 0 & 0 \\
1 & 1 & 0 & 0 
\end{bmatrix}. \] 
\end{rem}

\section{Main Results}

Given that it will be used in several of the subsequent results, hereinafter, $\alpha_{ij} := (i-j)\bmod{h}$, $\forall i,j \in \mathbb{Z}$ and $\forall h \in\mathbb{N}$. 

The following result was established by McDonald and Paparella \cite[Lemma 4.1]{mp2016} for nonsingular matrices; however, examining the proof reveals that the supposition is unnecessary.

\begin{lem} 
[{\cite[Lemma 4.1]{mp2016}}]
\label{lem:jchainlemma}
Let $A \in \mathsf{M}_n(\mathbb{C})$ and suppose that $A$ is of the form \eqref{cyclic_form}.  
\begin{enumerate}
\item If $\left\{ x_{\langle 0, j \rangle} \right\}_{j=1}^p$ is a right Jordan chain corresponding to $\lambda \in \sigma(A)$, where $x_{\left< 0, j \right>}$ is partitioned conformably with respect to  $A$ as 
\begin{align*}
x_{\left< 0, j \right>} = 
\begin{bmatrix} 
x_{1j} \\ 
\vdots \\ 
x_{hj} \end{bmatrix},~j \in \langle p \rangle,
\end{align*}
then, for $k \in R(h)$, the set 
\begin{align*}
\left\{
x_{\left< k, j \right>} :=
\begin{bmatrix} 
(\omega^k)^{\alpha_{1j}} x_{1j} \\ 
\vdots \\ 
(\omega^k)^{\alpha_{hj}} x_{hj} 
\end{bmatrix} 
\right\}_{j=1}^p
\end{align*}
is a right Jordan chain corresponding to $\lambda \omega^k$.
\item If $\left\{ y_{\left< j, 0 \right>} \right\}_{j = 1}^p$ is a left Jordan chain corresponding to $\lambda \in \sigma(A)$, where $y_{\left< j , 0 \right>}$ is partitioned conformably with respect to  $A$ as 
\begin{align*}
y_{\left< j , 0 \right>}^\top = 
\begin{bmatrix} 
y_{j1}^\top & 
\cdots & 
y_{jh}^\top \end{bmatrix},~j \in \langle p \rangle,
\end{align*}
then, for $k \in R(h)$, the set
\begin{align*}
\left\{  
y_{\left< j , k \right>}^\top :=
\begin{bmatrix} 
(\omega^k)^{\alpha_{j1}} y_{j1}^\top & 
\cdots & 
(\omega^k)^{\alpha_{jh}} y_{jh}^\top 
\end{bmatrix} \right\}_{j = 1}^p 										
\end{align*}
is a left Jordan chain corresponding to $\lambda \omega^k$.
\end{enumerate}
\end{lem}

\begin{rem}
Although McDonald and Paparella stated the aforementioned result for nonzero eigenvalues, the proof given is valid when $\lambda = 0$. However, if any Jordan canonical form of $A$ has a singular Jordan block of size $p$-by-$p$, then the result does not yield another Jordan block (it does, at least, yield a different Jordan chain for that eigenvalue).

For example, consider the $3$-cyclic matrix $A$ defined by 
\begin{align*}
A = 
\left[
\begin{array}
{*{1}{c}|*{3}{c}|*{2}{c}}
0 & 1 & 1 & 1 & 0 & 0   \\
\hline
0 & 0 & 0 & 0 & 1 & 0   \\
0 & 0 & 0 & 0 & 0 & 1   \\
0 & 0 & 0 & 0 & 1 & 1   \\
\hline
1 & 0 & 0 & 0 & 0 & 0   \\
1 & 0 & 0 & 0 & 0 & 0 
\end{array} \right].
\end{align*}
Since $B_1 = \begin{bmatrix} 4 \end{bmatrix}$, it follows that the nonzero eigenvalues of $A$ are $2^{2/3}$, $2^{2/3}\omega$, and $2^{2/3}\omega^2$. Furthermore, since
\[   
\left[
\begin{array}
{*{1}{c}|*{3}{c}|*{2}{c}}
0 & 1 & 1 & 1 & 0 & 0   \\
\hline
0 & 0 & 0 & 0 & 1 & 0   \\
0 & 0 & 0 & 0 & 0 & 1   \\
0 & 0 & 0 & 0 & 1 & 1   \\
\hline
1 & 0 & 0 & 0 & 0 & 0   \\
1 & 0 & 0 & 0 & 0 & 0 
\end{array} \right]
\left[ \begin{array}{r}
0 \\
\hline
0 \\ 
0 \\ 
0 \\
\hline 
1 \\ 
-1
\end{array} \right] 
= 
\left[ \begin{array}{r}
0 \\
\hline
1 \\ 
-1 \\ 
0 \\
\hline 
0 \\ 
0
\end{array} \right],
\]
and
\[  
\left[
\begin{array}
{*{1}{c}|*{3}{c}|*{2}{c}}
0 & 1 & 1 & 1 & 0 & 0   \\
\hline
0 & 0 & 0 & 0 & 1 & 0   \\
0 & 0 & 0 & 0 & 0 & 1   \\
0 & 0 & 0 & 0 & 1 & 1   \\
\hline
1 & 0 & 0 & 0 & 0 & 0   \\
1 & 0 & 0 & 0 & 0 & 0 
\end{array} \right]
\left[ \begin{array}{r}
0 \\
\hline
1 \\ 
-1 \\ 
0 \\
\hline 
0 \\ 
0
\end{array} \right],
\]
it follows that 
\[ 
\left\{
x_1 := \begin{bmatrix} 0 \\ 1 \\ -1 \\ 0 \\ 0 \\ 0 \end{bmatrix},\
x_2 := \begin{bmatrix} 0 \\ 0 \\ 0 \\ 0 \\ 1 \\ -1 \end{bmatrix}
\right\}
\]
is Jordan chain of length two associated with the eigenvalue zero. By Lemma \ref{lem:jchainlemma},
\[ 
\left\{
\hat{x}_1 
= 
\begin{bmatrix} 
(\omega)^{\alpha_{11}}\cdot 0       \\ 
(\omega)^{\alpha_{21}}\cdot 1       \\ 
(\omega)^{\alpha_{21}}\cdot (-1)    \\ 
(\omega)^{\alpha_{21}}\cdot 0       \\ 
(\omega)^{\alpha_{31}}\cdot 0       \\ 
(\omega)^{\alpha_{31}}\cdot 0 
\end{bmatrix} 
= \begin{bmatrix} 0 \\ \omega \\ -\omega \\ 0 \\ 0 \\ 0 \end{bmatrix},\
\hat{x}_2
= 
\begin{bmatrix} 
(\omega)^{\alpha_{12}}\cdot 0       \\ 
(\omega)^{\alpha_{22}}\cdot 0       \\ 
(\omega)^{\alpha_{22}}\cdot 0       \\ 
(\omega)^{\alpha_{22}}\cdot 0       \\ 
(\omega)^{\alpha_{32}}\cdot 1       \\ 
(\omega)^{\alpha_{32}}\cdot (-1) 
\end{bmatrix}
= \begin{bmatrix} 0 \\ 0 \\ 0 \\ 0 \\ \omega \\ -\omega \end{bmatrix}
\right\}
\]
is also a Jordan chain corresponding to zero, but any Jordan canonical form of $A$ has one Jordan block of order two and one Jordan block of order one. 
\end{rem}

The following results were presented by McDonald and Paparella \cite[Lemma 4.3 and Remark 4.4]{mp2016}. We repeat the results for the sake of completeness and to give more rigorous proofs that utilize the novel characterization of circulant matrices given in Theorem \ref{circ-char}.

\begin{lem}
[c.f.~{\cite[Lemma 4.3]{mp2016}}]
    \label{W-matrices}
If $k\in R(h)$ and $\ell \in \langle r \rangle$, then 
\begin{align*}
W_{k\ell}^1 
&:= 
\omega^k
\begin{bmatrix}
(\omega^k)^{\alpha_{1\ell}} \\
\vdots \\
(\omega^k)^{\alpha_{h\ell}}
\end{bmatrix}                                                       
\begin{bmatrix}
(\omega^k)^{\alpha_{\ell1}} & \ldots & (\omega^k)^{\alpha_{\ell h}}
\end{bmatrix}                                                       \\
&= \circulant(\omega^k, 1, (\omega^k)^{h-1}, \ldots, (\omega^k)^2)
\end{align*}
and 
\begin{align*}
W_{k\ell}^2 
&:= 
\begin{bmatrix}
(\omega^k)^{\alpha_{1\ell}} \\
\vdots \\
(\omega^k)^{\alpha_{h\ell}}
\end{bmatrix}
\begin{bmatrix}
(\omega^k)^{\alpha_{(\ell+1)1}} & \ldots & (\omega^k)^{\alpha_{(\ell+1) h}}
\end{bmatrix}                                                               \\
&= \circulant(\omega^k, 1, (\omega^k)^{h-1}, \ldots, (\omega^k)^2).    
\end{align*}
\end{lem}

\begin{proof}
The following facts are easily established:
\begin{align}
\alpha &\equiv \beta \bmod{h} \Longrightarrow (\omega^k)^\alpha = (\omega^k)^\beta,~k \in \mathbb{Z}	\label{omega}   \\
\alpha_{ij} &= (\alpha_{i \ell} + \alpha_{\ell j})\bmod{h},~\forall i,j,\ell \in \mathbb{Z}				\label{alpha3}	\\
\alpha_{i+1,j} &= \alpha_{i,j-1} = (\alpha_{ij} + 1) \bmod{h},~\forall i,j \in \mathbb{Z}.			    \label{alpha2}	
\end{align}
If $k\in R(h)$ and $\ell \in \langle r \rangle$, then
\begin{align*}
\left[W_{k\ell}^1 \right]_{ij} 
&= \omega^k(\omega^k)^{\alpha_{i\ell}} (\omega^k)^{\alpha_{\ell j}}                         \\
&= \omega^k(\omega^k)^{(\alpha_{i\ell} + \alpha_{\ell j})\bmod{h}}  \tag{by \eqref{omega}}  \\
&= \omega^k(\omega^k)^{\alpha_{ij}}                                 \tag{by \eqref{alpha3}} \\
&= (\omega^k)^{(\alpha_{ij} + 1)\bmod{h}}                           \tag{by \eqref{omega}}  \\
&= (\omega^k)^{\alpha_{i+1,j}}                                      \tag{by \eqref{alpha2}}
\end{align*}
If $(i,j)\sim(p,q)$, then
\[
\left[ W_{k\ell}^1 \right]_{ij} 
= (\omega^k)^{\alpha_{i+1,j}} 
= (\omega^k)^{\alpha_{p+1,q}} 
= \left[W_{k\ell}^1 \right]_{pq}, \]
so that, by Theorem \ref{circ-char}, $W_{k\ell}^1$ is a circulant. Since 
\[
\left[ W_{k\ell}^1 \right]_{1j} = (\omega^k)^{\alpha_{2j}} = (\omega^k)^{(2-j)\bmod{h}},
\]
it follows that 
\begin{align*}
\begin{pmatrix}
\left[ W_{k\ell}^1 \right]_{11}, 
\left[ W_{k\ell}^1 \right]_{12},
\left[ W_{k\ell}^1 \right]_{13},
\ldots, 
\left[ W_{k\ell}^1 \right]_{1h}
\end{pmatrix} 
=
\begin{pmatrix}
\omega^k, 
1,  
(\omega^k)^{h-1}, 
\ldots, 
(\omega^k)^2
\end{pmatrix},
\end{align*} 
as desired.

For the second claim, notice that 
\begin{align*}
\left[ W_{k\ell}^2 \right]_{ij} 
&= (\omega^k)^{\alpha_{i\ell}} (\omega^k)^{\alpha_{\ell+1,j}}                           \\
&= (\omega^k)^{\alpha_{i\ell}} (\omega^k)^{\alpha_{\ell,j-1}}   \tag{by \eqref{alpha2}} \\
&= (\omega^k)^{(\alpha_{i\ell} + \alpha_{\ell,j-1})\bmod{h}}    \tag{by \eqref{omega}}  \\
&= (\omega^k)^{\alpha_{i,j-1}}                                  \tag{by \eqref{alpha3}} \\
&= (\omega^k)^{\alpha_{i+1,j}}                                  \tag{by \eqref{alpha2}} \\
&= \left[ W_{k\ell}^1 \right]_{ij},
\end{align*}
as desired.
\end{proof}

\begin{lem}
[{c.f.~\cite[Remark 4.4]{mp2016}}]
\label{K-circulant}
If $k \in R(h)$ and  
\begin{equation}
\label{Ck}
C_k := \circulant(\omega^k, 1, (\omega^k)^{h-1}, \ldots, (\omega^k)^2) \in \mathsf{M}_h(\mathbb{C}),    
\end{equation}
then 
\[
\sum_{k=0}^{h-1}C_k = h K_h = \circulant(0, h, 0, \ldots, 0).
\] 
\end{lem}

\begin{proof}
Follows from the fact that circulant matrices are closed with respect to addition \cite[Theorem 3.24]{d1979} and the fact that 
\[
\sum_{k=0}^{h-1} (\omega^k)^p = h, 
\]
if $p = 0$ and 
\[
\sum_{k=0}^{h-1} (\omega^k)^p 
= \sum_{k=0}^{h-1} (\omega^p)^k 
= \frac{(\omega^p)^h - 1}{\omega^p - 1}
= \frac{(\omega^h)^p - 1}{\omega^p - 1}
= 0, 
\]
whenever $p \ne 0$.
\end{proof}

\begin{thm}
\label{thm:jchain}
Let $A \in \mathsf{M}_{n}(\mathbb{C})$ and let $P = \{ V_1, \ldots, V_h \}$ be a partition of $\langle n \rangle$.
\begin{enumerate}
[label = (\roman*)]
\item \label{cond(i)}
Suppose that, for every eigenvalue $\lambda \in \sigma(A)$ with corresponding right Jordan chain $\left\{ x_{\langle 0, j \rangle} \right\}_{j=1}^p$, and whenever the vector $x_{\left< 0, j \right>}$ is partitioned conformably with respect to  $P$ as 
\begin{align*}
x_{\left< 0, j \right>} = 
\begin{bmatrix} 
x_{1j} \\ 
\vdots \\ 
x_{hj} \end{bmatrix},~j \in \langle p \rangle,
\end{align*}
the set 
\begin{align*}
\left\{
x_{\left< k, j \right>} :=
\begin{bmatrix} 
(\omega^k)^{\alpha_{1j}} x_{1j} \\ 
\vdots \\ 
(\omega^k)^{\alpha_{hj}} x_{hj} 
\end{bmatrix} 
\right\}_{j=1}^p
\end{align*}
is a right Jordan chain corresponding to $\lambda \omega^k$ for every $k \in R(h)$. 
 
\item \label{cond(ii)}
Suppose that, for every eigenvalue $\lambda \in \sigma(A)$ with corresponding left Jordan chain $\left\{ y_{\left< j, 0 \right>} \right\}_{j = 1}^p$, and whenever the vector $y_{\left< j , 0 \right>}$ is partitioned conformably with respect to  $P$ as 
\begin{align*}
y_{\left< j , 0 \right>}^\top 
= 
\begin{bmatrix} 
y_{j1}^\top & 
\cdots & 
y_{jh}^\top 
\end{bmatrix},~j \in \langle p \rangle,
\end{align*}
the set
\begin{align*}
\left\{  
y_{\left< j , k \right>}^\top :=
\begin{bmatrix} 
(\omega^k)^{\alpha_{j1}} y_{j1}^\top & 
\cdots & 
(\omega^k)^{\alpha_{jh}} y_{jh}^\top 
\end{bmatrix} \right\}_{j = 1}^p 										
\end{align*}
is a left Jordan chain corresponding to $\lambda \omega^k$ for every $k \in R(h)$ 
\end{enumerate}
If the above hold, then $A$ is $h$-cyclic with partition $P$.
\end{thm}

\begin{proof}
By hypothesis, any Jordan canonical form of $A$ is of the form
\[ 
S^{-1} A S = 
\bigoplus_{i=1}^m \left( \bigoplus_{k=0}^{h-1} J_{n_i}\left(\lambda_i\omega^k \right) \right), 1 \le m < n. \]
For $i \in \langle m \rangle$, let
\[
D_i := \diag\left(0,\ldots,0,\overset{i}{\overbrace{\bigoplus_{j=0}^{h-1} J_{n_i}(\lambda_i)}},0,\ldots,0 \right)
\]
and \( A_i := S D_i S^{-1}\in \mathsf{M}_n(\mathbb{C})\), where $\diag(M_1,\ldots,M_k)$ denotes the block-diagonal matrix with block-diagonal entries $M_1,\ldots,M_k$.

For $k \in R(h)$, let $C_k$ be defined as in \eqref{Ck}. By Lemmas \ref{W-matrices} and \ref{K-circulant} and properties of the Hadamard product, notice that 
\begin{align*}
A_i  
&= 
\sum_{k=0}^{h-1} 
\left( 
\sum_{j=1}^{p_i} \lambda_i \omega^k x_{\left<k, j \right>} y_{\left< j , k \right>}^\top 
+ 
\sum_{j=1}^{p_i - 1} x_{\left<k, j \right>} y_{\left< j+1 , k \right>}^\top
\right) 						\\ 
&= 
\sum_{k=0}^{h-1} 
\left(
\sum_{j=1}^{p_i}
\lambda_i \omega^k 	
\begin{bmatrix} 
(\omega^k)^{\alpha_{1j}} x_{1j} \\ 
\vdots                          \\ 
(\omega^k)^{\alpha_{hj}} x_{hj} 
\end{bmatrix} 
\hspace{-3pt}
\begin{bmatrix} 
(\omega^k)^{\alpha_{j1}} y_{j1}^\top & \cdots & (\omega^k)^{\alpha_{jh}} y_{jh}^\top 
\end{bmatrix} \right.
+																	\\
& \qquad \left.
\sum_{j=1}^{p_i-1}
\begin{bmatrix} 
(\omega^k)^{\alpha_{1j}} x_{1j} \\ 
\vdots \\ 
(\omega^k)^{\alpha_{hj}} x_{hj} 
\end{bmatrix} 
\hspace{-3pt}
\begin{bmatrix} 
(\omega^k)^{\alpha_{j+1,1}} y_{j+1,1}^\top & \cdots & (\omega^k)^{\alpha_{j+1,h}} y_{j+1, h}^\top 
\end{bmatrix}  
\right)																	\\
&= 
\lambda_i
\sum_{k=0}^{h-1} \sum_{j=1}^{p_i} W_{kj}^1 \circ x_{\langle 0, j \rangle} y_{\langle j, 0 \rangle}^\top 
+
\sum_{k=0}^{h-1} \sum_{j=1}^{p_i-1} W_{kj}^2 \circ x_{\langle 0, j \rangle} y_{\langle j+1, 0  \rangle}^\top                    \\	
&= 
\lambda_i \sum_{j=1}^{p_i} \left( \sum_{k=0}^{h-1} C_k \circ x_{\langle 0, j \rangle} y_{\langle j, 0 \rangle}^\top \right) +
\sum_{j=1}^{p_i - 1} \left( \sum_{k=0}^{h-1} C_k \circ x_{\langle 0, j \rangle} y_{\langle j+1, 0 \rangle}^\top \right)         \\
&= 
\lambda_i \sum_{j=1}^{p_i} \left[ \left( \sum_{k=0}^{h-1} C_k \right) \circ x_{\langle 0, j \rangle} y_{\langle j, 0 \rangle}^\top \right] +				
\sum_{j=1}^{p_i - 1} 
\left[ \left( \sum_{k=0}^{h-1} C_k \right) \circ x_{\langle 0, j \rangle} y_{\langle j+1, 0 \rangle}^\top \right] 		        \\
&= 
\lambda_i h \sum_{j=1}^{p_i} K_h \circ x_{\langle 0, j \rangle} y_{\langle j, 0 \rangle}^\top + 
h \sum_{j=1}^{p_i-1} K_h \circ x_{\langle 0, j \rangle} y_{\langle j+1, 0 \rangle}^\top 								        \\
&= 
\lambda_i h \sum_{j=1}^{p_i}
\begin{bmatrix}
                    & x_{1j} y_{j2}^\top    &           &   	                    \\
                    &                       & \ddots    & 	                        \\
                    &                       &           & x_{(h-1)j} y_{jh}^\top	\\
x_{hj} y_{j1}^\top  &                       &           & 
\end{bmatrix} +															\\
& \qquad
h \sum_{j=1}^{p_i - 1}
\begin{bmatrix}
                        & x_{1j} y_{j+1,2}^\top &           &     	                    \\
                        &                       & \ddots    &     		                \\
                        &                       &           & x_{h-1,j} y_{j+1,h}^\top	\\
x_{hj} y_{j+1,1}^\top   &                       &           & 
\end{bmatrix} 															                                                        \\
&= 
\begin{bmatrix}
                & A_{12}^{(i)}  &           &   		            \\
                &               & \ddots    & 	                    \\
                &               &           &  A_{h-1,h}^{(i)}		\\
A_{h1}^{(i)}    &               &           & 
\end{bmatrix},							
\end{align*}
where 
\begin{align*}
A_{k,k+1}^{(i)} := 
\lambda_i h 
\sum_{j=1}^{p_i} x_{kj} y_{j,k + 1}^\top + 
h 
\sum_{j=1}^{p_i - 1} x_{kj} y_{j+1,k+1}^\top \in \mathsf{M}_{\vert V_k \vert \times \vert V_{k+1} \vert}(\mathbb{C}), 
\end{align*}
$k \in \langle h \rangle$, and, for convenience, $h + 1 = 1$. 

Clearly, the matrices $A_1,\dots,A_m$ are $h$-cyclic with partition $P$ and since $A = \sum_{i=1}^m A_i$, it follows that $A$ is $h$-cyclic with partition $P$.   
\end{proof}

\begin{ex} 
Consider the matrices
\[J = 
\left[
\begin{array}
{*{2}{c}|*{2}{c}}
0 & 1 & 0 & 0\\
0 & 0 & 0 & 0\\
\hline
0 & 0 & 0 & 1\\
0 & 0 & 0 & 0
\end{array}
\right]
\]
and
\[ 
S = 
\begin{bmatrix}
x & a & x & -a\\
y & b & y & -b\\
z & c & -z & c\\
w & d & -w & d
\end{bmatrix}\]
where $\det(S) = -4(ay-bx)(cw-dz) \ne 0$. 

If $A = S J S^{-1}$ and $P= \{ \{1,2\},\{3,4\}\}$, then $A$ satisfies the hypotheses of Theorem \ref{thm:jchain}, so $A$ is bipartite. Indeed, a calculation via a computer algebra system reveals that 
\begin{align*}
A                                                                      
=
\begin{bmatrix}
0                   & 0                 & \frac{wx}{cw-dz}  & \frac{-xz}{cw-dz} \\
0                   & 0                 & \frac{wy}{cw-dz}  & \frac{-yz}{cw-dz} \\
\frac{yz}{ay-bx}    & \frac{-xz}{ay-bx} & 0                 & 0                 \\
\frac{wy}{ay-bx}    & \frac{-xw}{ay-bx} & 0                 & 0
\end{bmatrix}.
\end{align*}
\end{ex}

The following result yields a complete characterization of the Jordan structure for invertible $h$-cyclic matrices.

\begin{thm}
If $A \in \mathsf{M}_n(\mathbb{C})$ is nonsingular and $P = \{ V_1,\ldots, V_h \}$ is a partition of $\langle n \rangle$, then $A$ is $h$-cyclic with partition $P$ if and only if the Jordan chains of $A$ satisfy conditions \ref{cond(i)} and \ref{cond(ii)} of Theorem \ref{thm:jchain}.  
\end{thm}

\begin{proof}
Follows by Lemma \ref{lem:jchainlemma} and Thereom \ref{thm:jchain}.
\end{proof}

\section{Jordan Chains of Singular \emph{h}-cyclic Matrices}

\begin{lem}
\label{lem:singular-spectrums}
If $A\in \mathsf{M}_n(\mathbb{C})$ is $h$-cyclic, then $A$ is singular if and only if at least one of the matrices $B_i$ is singular.
\end{lem}
\begin{proof}
In view of \eqref{A-to-the-h}, it follows that $\sigma(A^h) = \bigcup_{i=1}^h \sigma(B_i)$. Thus, $0\in\sigma(A^h)$ if and only if one of the submatrices $B_i$ of $A^h$ is singular, and the latter holds if and only if $0\in\sigma(A)$.
\end{proof}

\begin{lem}
\label{lem:circpiece}
Let $A\in \mathsf{M}_n(\mathbb{C})$ be $h$-cyclic and of the form \eqref{cyclic_form} with consecutive partition $P$. For $x \in \mathbb{C}^{\vert V_i \vert}$, where $i \in \langle h \rangle$, let $v \in \mathbb{C}^n$ be the conformably partitioned vector defined by 
\begin{align*}
    v_k = 
    \begin{cases}
    x, & k = i \\
    0, & k \ne i
    \end{cases}.
\end{align*}
If $p \in \mathbb{N}$, then 
\begin{align*}
    [A^p v]_k = 
    \begin{cases}
    B_{ip}x,    & k = \alpha^{-p}(i)     \\
    0,          & k \ne \alpha^{-p}(i).
    \end{cases}
\end{align*} 
\end{lem}

\begin{proof}
Proceed by induction on $p$. If $p=1$ and $i=1$, then
\( \alpha^{-p}(i) = \alpha^{-1}(1) = h\). 
Thus, $[A^p v]_k = [A v]_k = 0$, whenever $1 \le k < h$, and \( [A^p v]_h = [A v]_h = A_{h1}x \).   
Since 
\[ B_{ip} = B_{11} = \prod_{j=h}^h A_{\alpha^{j-1}(1),\alpha^{j}(1)} = A_{\alpha^{h-1}(1),\alpha^{h}(1)} = A_{h1}, \]
the result holds when $i = p = 1$.

If $1 < i \le h$ and $p=1$, then \( \alpha^{-p}(i) = \alpha^{-1}(i) = i-1 \). Thus, $[A^p v]_k = [A v]_k = 0$, whenever $k \ne i-1$ and \( [A^p v]_{i-1} = [A v]_{i-1} = A_{i-1,i}x \). Since 
\[ B_{ip} = B_{i1} = \prod_{j=h}^h A_{\alpha^{j-1}(i),\alpha^{j}(i)} = A_{\alpha^{h-1}(i),\alpha^{h}(i)} = A_{i-1,i}, \]
the result holds when $1 < i \le h$ and $p = 1$, which completes the base-case.

For the induction-step, assume that 
\begin{align*}
    [A^p v]_k = 
    \begin{cases}
    B_{ip}x,    & k = \alpha^{-p}(i)     \\
    0,          & k \ne \alpha^{-p}(i),
    \end{cases}
\end{align*}
for some positive integer $p$. 

If $\alpha^{-p}(i) = 1$, then $\alpha^{-(p+1)}(i) = h$ and 
\begin{align*}
[A^{p+1} v]_{\alpha^{-(p+1)}(i)} 
&= [A (A^p v)]_h                                                                                                                                \\ 
&= A_{h1} B_{ip}x                                                                                                                               \\
&= A_{h1} \left( \prod_{j=h+1-p}^h A_{\alpha^{j-1}(i),\alpha^{j}(i)}x \right)                                                                   \\
&= A_{\alpha^{-(p+1)}(i), \alpha^{-p}(i)} \left( \prod_{j=h+1-p}^h A_{\alpha^{j-1}(i),\alpha^{j}(i)}x \right)                                   \\
&= A_{\alpha^{h-(p+1)}(i), \alpha^{h-p}(i)} \left( \prod_{j=h+1-p}^h A_{\alpha^{j-1}(i),\alpha^{j}(i)}x \right) \tag{$\alpha^h = \varepsilon$}  \\
&= \left( \prod_{j=h-p}^h A_{\alpha^{j-1}(i),\alpha^{j}(i)}x \right)                                                                            \\
&= B_{i,p+1} x,
\end{align*} 
as desired.

If $1 < \alpha^{-p}(i) \le h$, then $\alpha^{-(p+1)}(i) = i-1$ and 
\begin{align*}
[A^{p+1} v]_{\alpha^{-(p+1)}(i)} 
&= [A (A^p v)]_{i-1}                                                                                                                                \\ 
&= A_{i-1,i} B_{ip}x                                                                                                                               \\
&= A_{i-1,i} \left( \prod_{j=h+1-p}^h A_{\alpha^{j-1}(i),\alpha^{j}(i)}x \right)                                                                   \\
&= A_{\alpha^{-(p+1)}(i), \alpha^{-p}(i)} \left( \prod_{j=h+1-p}^h A_{\alpha^{j-1}(i),\alpha^{j}(i)}x \right)                                   \\
&= A_{\alpha^{h-(p+1)}(i), \alpha^{h-p}(i)} \left( \prod_{j=h+1-p}^h A_{\alpha^{j-1}(i),\alpha^{j}(i)}x \right) \tag{$\alpha^h = \varepsilon$}  \\
&= \left( \prod_{j=h-p}^h A_{\alpha^{j-1}(i),\alpha^{j}(i)}x \right)                                                                            \\
&= B_{i,p+1} x,
\end{align*} 
which completes the induction-step. The entire result now follows by the principle of mathematical induction.
\end{proof}

\begin{thm}
\label{thm:zero-j-chains}
Let $A\in \mathsf{M}_n(\mathbb{C})$ be $h$-cyclic with consecutive partition $P$. Suppose that $B_i$ is singular. If $x\in\nullspace(B_i),\ x\neq 0$, then there is a positive integer $p \le h$ such that $B_{ip}x = 0$ and $B_{iq}x\neq 0$ for all $1 \leq q < p$. Furthermore, any Jordan canonical form of $A$ has a $p\times p$ singular Jordan block.
\end{thm}

\begin{proof}
Let $v \in \mathbb{C}^n$ be the conformably partitioned vector defined by 
\begin{align*}
    v_k = 
    \begin{cases}
    x, & k = i \\
    0, & k \ne i
    \end{cases}.
\end{align*}

By Lemma \ref{lem:circpiece} and the assumption that $x\in\nullspace(B_i)$, we have
\begin{align*}
    [A^h v]_k 
    =  
    \begin{cases}
    B_{ih}x,    & k = \alpha^{-h}(i)    \\
    0,          & k \ne \alpha^{-h}(i)
    \end{cases}                             
    =  
    \begin{cases}
    B_ix,   & k = i    \\
    0,      & k \ne i,
    \end{cases}
    = 0.
\end{align*}
Thus, the set
\[ \begin{Bmatrix} k \in \mathbb{N} \mid (A - 0I_n)^k v = A^k v = 0 \end{Bmatrix} \ne \emptyset \]
and contains a minimal element---say $p$---such that $p \le h$ and $A^q v \ne 0$ whenever $1 \le q < p$. Hence, in view of Lemma \ref{lem:circpiece}, $B_{ip}x = 0$ and $B_{iq}x \ne 0$.

If $w_k := A^{p-k} v$, then $\{ w_1,\dots,w_p \}$ is a Jordan chain corresponding to the eigenvalue zero. Hence, any Jordan canonical form of $A$ has a $p\times p$ singular Jordan block.
\end{proof}

\begin{rem}
For $i \in \langle h \rangle$, let $\mathcal{B}_i = \left\{ b_1^{(i)},\ldots,b_{d_i}^{(i)} \right\}$ be a basis for $\nullspace(B_i)$. In Theorem \ref{thm:zero-j-chains}, for $x \in \nullspace(B_i)$, if $p_x$ denotes the length of the Jordan chain corresponding to $x$, then 
\[ p_x = \max \left\{ p_1^{(i)},\ldots,p_{d_i}^{(i)} \right\}, \]
where $p_k^{(i)}$ is the length of the Jordan chain corresponding to $b_k^{(i)}$, $k \in \{1,\ldots,d_i\}$. Thus, it suffices to only consider elements of $\mathcal{B}_i$ to determine the Jordan chains of $A$ corresponding to the eigenvalue zero. 
\end{rem}

\begin{ex}
Consider the 3-cyclic matrix
\[A = \begin{bmatrix}
    0 & A_{12} & 0\\
    0 & 0 & A_{23}\\
    A_{31} & 0 & 0
\end{bmatrix}\]
where $A_{12}$ is the all-ones matrix of order four, $A_{31} = I_4$, and 
\[
A_{23} = \begin{bmatrix}
1 & 0 & 0 & 0\\
1 & 0 & 0 & 0\\
1 & 0 & 0 & 0\\
0 & 0 & 0 & 1\\
\end{bmatrix}. 
\]
The \verb+jordan+ command in MATLAB reveals that a Jordan canonical form of $A$ is 
\[
\left[
\begin{array}
{*{3}{c}|*{2}{c}|*{2}{c}|*{1}{c}|*{1}{c}|*{3}{c}}
0 & 1 & 0 & 0 & 0 & 0 & 0 & 0 & 0 & 0 & 0 & 0   \\ 
0 & 0 & 1 & 0 & 0 & 0 & 0 & 0 & 0 & 0 & 0 & 0   \\ 
0 & 0 & 0 & 0 & 0 & 0 & 0 & 0 & 0 & 0 & 0 & 0   \\ 
\hline                                          
0 & 0 & 0 & 0 & 1 & 0 & 0 & 0 & 0 & 0 & 0 & 0   \\ 
0 & 0 & 0 & 0 & 0 & 0 & 0 & 0 & 0 & 0 & 0 & 0   \\ 
\hline
0 & 0 & 0 & 0 & 0 & 0 & 1 & 0 & 0 & 0 & 0 & 0   \\ 
0 & 0 & 0 & 0 & 0 & 0 & 0 & 0 & 0 & 0 & 0 & 0   \\ 
\hline
0 & 0 & 0 & 0 & 0 & 0 & 0 & 0 & 0 & 0 & 0 & 0   \\ 
\hline
0 & 0 & 0 & 0 & 0 & 0 & 0 & 0 & 0 & 0 & 0 & 0   \\ 
\hline
0 & 0 & 0 & 0 & 0 & 0 & 0 & 0 & 0 & 2^{2/3} & 0 & 0   \\ 
\hline
0 & 0 & 0 & 0 & 0 & 0 & 0 & 0 & 0 & 0 & 2^{2/3}\omega & 0   \\ 
\hline
0 & 0 & 0 & 0 & 0 & 0 & 0 & 0 & 0 & 0 & 0 & 2^{2/3}\omega^2      
\end{array} \right].
\]
A straightforward calculation reveals that
\[ 
B_1 = B_3 = \begin{bmatrix}
3 & 0 & 0 & 1\\
3 & 0 & 0 & 1\\
3 & 0 & 0 & 1\\
3 & 0 & 0 & 1
\end{bmatrix} \]
and  
\[ 
B_2 = 
\begin{bmatrix}
1 & 1 & 1 & 1\\
1 & 1 & 1 & 1\\
1 & 1 & 1 & 1\\
1 & 1 & 1 & 1
\end{bmatrix}. \]

If 
\[ 
x = \begin{bmatrix}
-1/3 & 0 & 0 & 1
\end{bmatrix}^\top, \]
\[ 
y = \begin{bmatrix}
0 & 1 & 0 & 0
\end{bmatrix}^\top, \]
and 
\[ z= \begin{bmatrix}
0 & 0 & 1 & 0
\end{bmatrix}^\top, \]
then $\mathcal{B}_1 = \{x,y,z\}$ is a basis for $\nullspace(B_1)$. It can then be shown that $Av_x\neq0, A^2v_x\neq0$, and $A^3v_x=0$ while $A^2v_y=0$ and $A^2v_z=0$ with $Av_y\neq0$ and $Av_z\neq0$. As a consequence of Theorem \ref{thm:zero-j-chains}, the matrix $B_1$ contributes singular Jordan blocks of size three and two to the Jordan canonical form of $A$. 

On the other hand, the matrices $B_1$ and $B_3$ share the same null space but the vectors $v_x, v_y$ and $v_z$ corresponding to $B_3$ differ from those of $B_1$ in the positions of their non-zero entries and we find that $B_3$ contributes Jordan chains of length one and two. The basis of $\nullspace (B_2)$ produces three chains of length one.  

Thus, there is potential for redundancy in producing Jordan chains via Theorem \ref{thm:zero-j-chains}. However, calculating the \emph{Weyr characteristic} of $A$ with respect to zero (see, e.g., Horn and Johnson \cite[Lemma 3.1.18]{hj2013}) in conjunction with Theorem \ref{thm:zero-j-chains} yields an accurate procedure. 
\end{ex}

\section{Acknowledgements}

The second author thanks Daniel Szyld for his talk at the 2019 Meeting of the International Linear Algebra Society which provided the impetus for this work and Michael Tsatsomeros for pointing out that Lemma 4.1 of McDonald and Paparella \cite{mp2016} generalized the result for bipartite graphs. 

\bibliographystyle{abbrv}
\bibliography{chains}

\begin{thebibliography}{1}

\bibitem{br1991}
R.~A. Brualdi and H.~J. Ryser.
\newblock {\em Combinatorial matrix theory}, volume~39 of {\em Encyclopedia of
  Mathematics and its Applications}.
\newblock Cambridge University Press, Cambridge, 1991.

\bibitem{d1979}
P.~J. Davis.
\newblock {\em Circulant matrices}.
\newblock A Wiley-Interscience Publication. John Wiley \& Sons, New
  York-Chichester-Brisbane, 1979.

\bibitem{hj2013}
R.~A. Horn and C.~R. Johnson.
\newblock {\em Matrix analysis}.
\newblock Cambridge University Press, Cambridge, second edition, 2013.

\bibitem{mp2016}
J.~J. McDonald and P.~Paparella.
\newblock Jordan chains of {$h$}-cyclic matrices.
\newblock {\em Linear Algebra Appl.}, 498:145--159, 2016.

\bibitem{m1966}
L.~Mirsky.
\newblock An inequality for characteristic roots and singular values of complex
  matrices.
\newblock {\em Monatsh. Math.}, 70:357--359, 1966.

\bibitem{n2017}
V.~Nikiforov.
\newblock Hypergraphs and hypermatrices with symmetric spectrum.
\newblock {\em Linear Algebra Appl.}, 519:1--18, 2017.

\end{thebibliography}
\end{document}